\documentclass[12pt,a4paper]{article}
\pdfoutput=1
\usepackage{amsthm,bm,amsmath,amssymb,epsfig,multicol,multirow}
\usepackage{natbib}
\usepackage{graphics,rotating,graphicx}
\usepackage{array,multirow,caption2}
\usepackage{amssymb}
\usepackage{amsthm,url}
\usepackage{bm,amsmath,multicol}
\usepackage{booktabs}

\newcommand{\bB}{\mathbf{B}}
\newcommand{\bW}{\mathbf{W}}
\newcommand{\bS}{\mathbf{S}}
\newcommand{\bD}{\mathbf{D}}

\newcommand{\bSigma}{\bm{\Sigma}}
\newcommand{\bmu}{\bm{\mu}}
\newcommand{\bX}{\mathbf{X}}
\newcommand{\bx}{\mathbf{x}}
\newcommand{\bz}{\mathbf{z}}
\newcommand{\xbar}{\overline{\mathbf{x}}}

\newcommand{\be}{\mathbf{e}}
\newcommand{\bv}{\mathbf{v}}

\newcommand{\ehat}{\widehat{\mathbf{e}}}
\newcommand{\phat}{\widehat{p}}
\newcommand{\lambdahat}{\widehat{\lambda}}
\newcommand{\alphahat}{\widehat{\alpha}}
\newcommand{\mB}{\mathcal{B}}
\newcommand{\mS}{\mathcal{S}}
\newcommand{\mT}{\mathcal{T}}
\newcommand{\mC}{\mathcal{C}}

\newcommand{\Cov}{\text{Cov}}
\newcommand{\IF}{\text{IF}}
\newcommand{\SIF}{\text{SIF}}
\newcommand{\EIF}{\text{EIF}}

\newtheorem{lemma}{Lemma}
\newtheorem{condition}{Condition}
\newtheorem{theorem}{Theorem}

\title{IFs for LDA}

\begin{document}

\noindent{\Large \textbf{Influence functions for Linear Discriminant Analysis: Sensitivity analysis and efficient influence diagnostics}}



\vspace{0.5cm}

\noindent LUKE A. PRENDERGAST and JODIE A. SMITH

\noindent \textit{Department of Mathematics and Statistics, La Trobe
University, Melbourne, Australia, 3086}

\vspace{0.5cm}

\noindent\textbf{ABSTRACT.  Whilst influence functions for linear discriminant analysis (LDA) have been found for a single discriminant when dealing with two groups, until now these have not been derived in the setting of a general number of groups.  In this paper we explore the relationship between Sliced Inverse Regression (SIR) and LDA, and exploit this relationship to develop influence functions for LDA from those already derived for SIR.  These influence functions can be used to understand robustness properties of LDA and also to detect influential observations in practice.  We illustrate the usefulness of these via their application to a real data set.}

\vspace{0.5cm}

\noindent \textit{Key words:} dimension reduction; eigen-subspace; influential observations; linear discriminant analysis; sliced inverse regression

\section{Introduction}

With origins dating back to \cite{Fisher1936use}, linear discriminant analysis (LDA) is one of the most widely used statistical tools for supervised classification of subjects into a finite number of sub-populations.  The simplicity of LDA sees it taught in many university programs and its availability within popular statistics computing software ensures that it is a popular method among practitioners.  

Introduced by \cite{LI91}, Sliced Inverse Regression (SIR) is a dimension reduction technique that, under mild conditions and for a general model framework, can be used to reduce the dimension of the given predictor space given.  SIR and LDA are strongly connected with \cite{LI00} noting only a difference in scaling evident. As a result new influence functions for LDA are then also available using results available for SIR \citep[e.g.][]{PR05,PR&SM10}.  Unlike existing influence functions for LDA, these new influence functions are applicable in the setting of multiple directions and are also based on the discriminants themselves and not just the directions used to define them.  More details in this regard will follow later.  Influence functions for methods such as LDA can be used to devlop and assess robust estimators \citep[e.g., see][who consider robust inverse regression]{dong2015robust}. 

In Section 2 we provide some background concepts for both LDA and influence functions before providing influence functions for LDA in Section 3.  An application of sample-based versions to a real data set  are considered in Section 4.  Theoretical derivations including some discussion on the link between SIR and LDA can be found in the Appendix.

\section{Background material and concepts}

In this section we provide some background material for LDA and influence functions that will  be useful later.

\subsection{Linear discriminant analysis}\label{sect:LDA}

Throughout we will use notations consistent with Section 11.7 from \cite{JO&WI01} which concerns Fisher's method for discrimination between several populations.  We assume that there are $g$ populations and we let $\bx_{i1}, \ldots, \bx_{in_i}$ denote $n_i$ $p$-dimensional column vectors observed from the $i$th population.  The sample mean of the $\bx_{ij}$'s within the $i$th population is $\xbar_i$ and the sample mean of all $\bx_{ij}$'s across all populations is $\xbar$.  For use later we consider two estimators of the covariance matrix for the $\bx$'s within the $i$th group.  The first is the empirical estimate given as $\widehat{\bSigma}_i=n_i^{-1}\sum^{n_i}_{j=1}(\bx_{ij}-\xbar_i)(\bx_{ij}-\xbar_i)^\top$ and the second is the usual sample covariance matrix estimator $\bS_i=n_i\widehat{\bSigma}_i/(n_i-1)$.

Variation between the $g$ groups is measured by
\begin{equation}
\widehat\bB=\sum^g_{i=1}n_i(\xbar_i-\xbar)(\xbar_i-\xbar)^\top\label{B}
\end{equation}
and a pooled measure for within groups variation is
\begin{equation}
\widehat\bW = \sum^g_{i=1}(n_i-1)\bS_i=\sum^g_{i=1}n_i\widehat{\bSigma}_i.\label{W}
\end{equation}

Fisher's linear discriminant directions are obtained from the matrix $\widehat\bD=\widehat\bW^{-1}\widehat\bB$ which has maximum rank $g-1$.  Let $s$ denote the rank of $\widehat\bD=\widehat\bW^{-1}\widehat\bB$ and $\lambdahat_1\geq\ldots\geq\lambdahat_s>0$ denote the ordered nonzero eigenvalues corresponding to the eigenvectors $\ehat_1,\ldots,\ehat_s$ of $\widehat\bD$ scaled such that $\ehat_j^\top\bS_p\ehat_j=1$.  Here $\bS_p=\widehat\bW/(n_1+\ldots+n_g-g)$ is the pooled sample variance estimate of the assumed common covariance matrix within each group.

\subsection{The influence function}

Let $F$ denote an arbitrary distribution function for which a parameter, denoted $\theta$, is of interest.  Let $\mT$ denote a statistical functional such that when applied to the distribution function $F$, is $\mT(F)=\theta$.  Further, let $F_n$ denote the empirical distribution function associated with a sample of size $n$ from $F$ so that $F_n\rightarrow F$ for increasing $n$.  Then an estimator of $\theta$ is $\widehat{\theta}=\mT(F_n)$ with $\mT(F_n)\rightarrow \theta$ as $n\rightarrow \infty$.  An example statistical functional is that of the mean.  Here $\mT(F)=\int x dF=\mu$ and $\mT(F_n)=n^{-1}\sum^n_{i=1}x_i=\overline{x}$ (the sample mean).

Let $\Delta_{x_0}$ denote the Dirac distribution function that puts all probability mass at the point $x_0$ (which may be a vector in the multivariate setting or a scalar).  For $x_0$ being the contaminant, the contaminated distribution is
\begin{equation}\label{Fe}
F_\epsilon = (1-\epsilon)F+\epsilon \Delta_{x_0}
\end{equation}
where $\epsilon \in [0,1]$.  When $\epsilon$ is small, we seek to understand by how much a small amount of contamination influences $\mT(F_\epsilon)$ so that it differs from $\mT(F)$.  To study the robustness properties of estimators, \cite{hampel1974influence} introduced the influence function (IF) defined to be
\begin{equation}
    \IF(\mT,F;x_0)=\lim_{\epsilon \downarrow 0}\frac{\mT(F_\epsilon)-\mT(F)}{\epsilon}=\frac{\partial}{\partial\epsilon}\mT(F_\epsilon)\Big|_{\epsilon=0}\label{IF}.
\end{equation}

When thought of as a power series expansion, we have
\begin{equation}
    \mT(F_\epsilon)=\mT(F)+\epsilon \IF(\mT,F;x_0) + O(\epsilon^2)\label{power}
\end{equation}
which helps to conceptualize the relevance of the IF to studying robustness properties.  For example, if $\IF(\mT,F;x_0)=0$ then the introduction of contaminant $x_0$ has little or no influence on the estimator functional.  On the contrary, it is also possible to use the IF to determine the types of $x_0$ that exert large influence (have a large IF).

The IF is not only useful for studying robustness properties as it can also be employed in practice to measure the influence of an observation on a sample estimate.  Let $x_1,\ldots,x_n$ denote a sample.  Then the sample IF (SIF) for the $i$th observation is
\begin{equation}
    \SIF(\mT,F_n;x_i)=(n-1)\left[\mT(F_n)-\mT(F_{n,(i)})\right]
\end{equation}
where $F_{n,(i)}$ is the empirical distribution function for the sample excluding $x_i$.  That is, the SIF is proportional to the difference in estimates with and without the $i$th observation.  The empirical IF (EIF) results directly from a closed form expression for the IF but replaces $x_0$ with $x_i$ and parameters with their estimates.  It is important to note that, for moderate to large $n$,
$$\EIF(\mT,F_n;x_i)\approx \SIF(\mT,F_n;x_i).$$

The above concepts are well articulated by \cite{CR85} who studies the IFs, SIFs and EIFs of principal components.  As an example, consider the covariance matrix estimator and a sample of $p$ dimensional vectors from $F$ denoted $\mathbf{x}_1,\ldots,\mathbf{x}_n$, and let $\mT(F)=\bmu$ and $\mC(F)=\bSigma$ denote the mean vector and covariance matrix.  Then $\IF(\mC,F;x_0)=(\bx_0-\bmu)(\bx_0-\bmu)^\top-\bSigma$ and $\SIF(\mC,F_n;\bx_i)=(n-1)\left[\widehat\bSigma - \widehat\bSigma_{(i)}\right]$ where $\widehat\bSigma=\mC(F_n)$ and $\widehat\bSigma_{(i)}=\mC(F_{n,(i)})$.  Now, using the definition of $\IF(\mC,F;x_0)$, we can obtain the EIF
\begin{equation}\label{EIF.C}
    \EIF(\mC,F_n;\mathbf{x}_i)=(\bx_i-\overline{\bx})(\bx_i-\overline{\bx})^\top-\widehat\bSigma
\end{equation}
where $\overline{\mathbf{x}}$ is the sample mean.  We can verify that $\EIF(\mC,F_n;\mathbf{x}_i)\approx \SIF(\mC,F_n;\mathbf{x}_i)$ by noting that, see e.g. \cite{prendergast2007implications}, $\widehat{\bSigma}_{(i)}=[n/(n - 1)]\widehat{\bSigma}-[n/(n-1)^2](\bx_i-\overline\bx)(\bx_i-\overline\bx)^\top$.  We can use this to show that
\begin{equation}
    \SIF(\mC,F_n;\mathbf{x}_i)=\frac{n}{(n-1)}(\bx_i-\overline{\bx})(\bx_i-\overline{\bx})^\top-\widehat\bSigma
\end{equation}
so that the EIF from \eqref{EIF.C} is approximately equal to the SIF.  This then speaks to the dual purpose in what is to follow.  Firstly, to introduce influence functions that may be used to study the robustness properties of LDA and secondly to create influence diagnostics that can be used in practice.

\section{Influence functions for linear discriminants}

For simplicity in defining functionals, it is convenient to move slightly away from the group-specific indexing considered in the previous section.  Here we let $n=n_1+\ldots+n_g$ denote the total sample size across all groups and let $\bx_j$ denote the $j$th observed $\bx$ from the $n$ observed vectors.  We also let $y_j$ denote the indicator for group membership for $\bx_i$ where $y_j\in \{1,\ldots,g\}$.  Let $\phat_i = n_i/n$ denote the proportion of $\bx$'s in the $i$th group.  Then re-scaled versions of $\bB$ and $\bW$ from \eqref{B} and \eqref{W} can be defined as
\begin{equation}
\widehat\bB_n = \frac{1}{n}\widehat\bB = \sum^g_{i=1}\phat_i(\xbar_i-\xbar)(\xbar_i-\xbar)^\top\;\;\text{and}\;\;\widehat\bW_n = \frac{1}{n}\widehat\bW = \sum^g_{i=1}\phat_i\widehat{\bSigma}_i. \label{Bn_and_Wn}
\end{equation}

Clearly, we can still arrive at the discriminant matrix $\widehat\bD$ from the previous section by
\begin{equation}
\widehat\bD = \widehat\bW_n^{-1} \widehat\bB_n \label{D}
\end{equation}
and our motivation for using this notation becomes clear in the next section.

\subsection{Functionals for the linear discriminant estimators}

In order to define functionals for our discriminant estimators, we need to introduce $F$ - the distribution function for our `population'.  Here the population consists of $g$ non-overlapping subpopulations and we can imagine that these subpopulations have distribution functions $F_1,\ldots,F_g$.  We now consider a random pair $(\bX, Y)\sim F$ which consists of the random covariate vector $\bX$ and random group membership indicator $Y\in \{1,\ldots,g\}$ where $p_i = P(Y = i)$ for $i=1,\ldots,g$.  Further, we define the following:
\begin{equation*}
\bmu_i = E(\bX|Y=i),\; \bmu =  E(\bX) = \sum^g_{i=1}p_i\bmu_i,\;\bSigma_i =  \Cov(\bX|Y=i),\; \bSigma =  \Cov(\bX)
\end{equation*}
for which the empirical estimates for $\bmu_i$, $\bmu$ and $\bSigma_i$ are $\xbar_i$, $\xbar$ and $\widehat{\bSigma}_i$ introduced in Section \ref{sect:LDA}.  This leads us to note that $\widehat\bB_n$ and $\widehat\bW_n$ are the empirical estimates to $\Cov[E(\bX|Y)]$ and $E[\Cov(\bX|Y)]$ respectively, both of which are fundamental to the development of SIR \citep{LI91} and related methods.  The link between LDA and SIR can be found in Chapter 14 of \cite{LI00} and further details are provided in Appendix \ref{sect:LDA_and_SIR}.

Let $\mathcal{B}$ denote the functional for the estimator associated with $\bB_n$ such that $\mathcal{B}(F_n)=\bB_n$ and where $F_n$ denotes the empirical distribution functional for $\{(\bx_1,y_1),\ldots,(\bx_n,y_n)\}$.  Similarly we let $\mathcal{W}$ denote the functional for the estimator associated with $\bW_n$.  From these functionals we can define the functional for the discriminant matrix estimator as
\begin{equation}
\mathcal{D}(F) = \left[\mathcal{W}(F)\right]^{-1}\mathcal{B}(F)
\end{equation}
where $\mathcal{D}(F_n)=\widehat\bD$.  Then let $e_1,\ldots,e_s$, where $s$ is the rank of $D(F)$, be the functionals for the discriminant directions which are scaled eigenvectors of $D(F)$ and where $e_j(F)=\be_j$ $(j=1,\ldots,s)$.  We also let $l_j$ be the eigenvalue functionals with $l_j(F)=\lambda_j$.

\subsection{Existing influence functions}

In the case of two subpopulations, i.e. $g=2$, the rank of $\widehat\bD$ is one and there is a single discriminant direction which is proportional to $\bS^{-1}(\overline\bx_1-\overline\bx_2)$.  In studying robust versions of LDA, \cite{croux2001robust} derive the IF for this discriminant direction as well as for a robust version of the direction that has a bounded IF.  Not long after, \cite{pires2002partial} discussed influence functions more broadly in the setting of subpopulations, and also gave the IF for the single discriminant direction when $g=2$.  In the next section we provide the IFs for the directions for the general case of $g\geq 1$ arising from the eigen-decomposition of $\bD$.  We then turn our attention to an improved IF that is simpler in expression and offers some notable advantages.

\subsection{Influence functions for discriminant directions}

We start by providing the IF for the LDA eigenvalue estimators and discriminant directions.

\begin{theorem}\label{theorem:IFs}
Suppose that $y_0=k$ so that the contaminant, $\bx_0$, is in the $k$th subpopulation.  Then the IF for the $j$th LDA eigenvalue estimator is equal to
\begin{equation*}
    \IF(l_j,F;\bx_0)=z_j^2-(1+\lambda_j)\overline z_{jk}^2
\end{equation*}
where $z_j=\be_j^\top(\bx_0-\bmu)$ and $\overline{z}_{jk}=\be_j^\top(\bx_0-\bmu_k)$.  The IF for the $j$th discriminant direction is of the form
\begin{equation*}
    \IF(e_j,F;\bx_0)=\frac{\IF(l_j,F;\bx_0)}{2\sqrt{1+\lambda_j}}\be_j+\sqrt{1+\lambda_j}\IF(v_j,F;\bx_0)
\end{equation*}
where $\IF(v_j,F;\bx_0)$ is the IF for the $j$th SIR direction, which is given in \eqref{IFvj} of the Appendix, where $\alpha_j=\lambda_j/(1+\lambda_j)$, $w_j=z_j/\sqrt{1+\lambda_j}$, $\overline w_{jk}=\overline z_{jk}/\sqrt{1+\lambda_j}$ and $\bv_j=\be_j/\sqrt{1+\lambda_j}$ for $j=1,\ldots,s$.
\end{theorem}

While the IF for the eigenvalues is simple, the same cannot be said for the LDA discriminant directions.  The IF for the $j$th direction is unnecessarily complicated and we find it more convenient to leave it in terms of the SIR IF in Theorem \ref{theorem:IFs}.  To understand why this is the case, we note that the IF for the $j$th direction is of the form
$$\IF(e_j,F;\bx_0)=c_1\be_j+\sum^s_{r=1,r\neq j}c_{rj}\be_r + c_2\bSigma^{-1}(\bx_0-\bmu)$$
for $c_1,c_2,c_{rj}\in \mathbb{R}$.  Note that $s$ is the rank of $\bD$ so that the discriminant directions $\be_1,\ldots,\be_s$ all contain information regarding differences in the subpopulations.  When considering \eqref{power}, the terms proportional to these directions are not necessarily indicative of a problem since, if the $j$th direction is perturbed towards these other $\be_r$'s, then a direction with information regarding the separation is still found.  This then leads us to our preferred IF which takes into account the span of the discriminants. 

\subsection{Influence functions for spans of discriminants}\label{section:IFs_for_span}

Consider the population discriminant directions $\be_1,\ldots,\be_s$.  Then, for an arbitrary $\bx$ sampled from the population, the linear discriminants are given as $\be_1^\top\bx,\ldots,\be_s^\top\bx$.  From an influence perspective it is important to note the following:
\begin{description}
\item[(i)] When $s>1$, we should not focus our attention on an individual $\be_j$, but rather all $\be_j$'s collectively since any influence on a single direction may be nullified by a change in the others (e.g. an extreme example is when contamination simply results in two directions switching place in the order determined by the eigenvalues).
\item[(ii)] It also seems practical to take into account the correlation amongst the variables in $\bx$ since large changes in a component of an $\be_j$ may not be large at all given the correlation structure (e.g. an extreme example of this is when two variables are highly correlated; here a big change to the corresponding component for one of these variables in $\be_j$ may be nullified by a change in the corresponding component of another).
\end{description}

These points motivated the development of influence functions that account for correlation amongst variables by \cite{PR&SM10} for dimension reduction methods such as SIR and \cite{PR&LI11} for principal component analysis.  Importantly, given that the LDA directions are scalar proportional to the SIR directions, it suffices to adapt the measure by \cite{PR&SM10} to SIR for a discrete response and then to express the IF in terms of the LDA eigenvalues and directions.

  Following \cite{PR&SM10}, we start by considering the empirical setting where $(\bx_1,y_1),\ldots,(\bx_n,y_n)$ denotes the total of $n$ observations where the $y_i$'s specify the subpopulation to which the observation belongs.  Let $\mathbf{X}_n$ denote the $n\times p$ matrix whose $i$th row is equal to $\bx_i^\top$ and let $\widehat{\mathbf{E}}=[\widehat\be_1,\ldots,\widehat\be_s]$ be the matrix whose columns are the LDA discriminant directions. Let $\mathbf{E}$ denote the matrix of true directions estimated by $\widehat{\mathbf{E}}$.  It is not necessary for each $\widehat\be_j$ to be mapped directly as an estimate to the $j$th column of $\mathbf{E}$ since, for e.g., if the column span of $\widehat{\mathbf{E}}$ is equal to the column span of $\mathbf{E}$, then LDA has succeeded. We are also interested in how $\mathbf{X}_n\widehat{\mathbf{E}}$ (whose $j$ columns consist of the linear sample discriminants $\widehat\be_j^\top\bx_1,\ldots,\widehat\be_j^\top\bx_n$) differs from $\mathbf{X}_n\mathbf{E}$.  As was considered by \cite{LI91} in the context of SIR, this can be accomplished by considering the squared correlations between $\mathbf{X}_n\widehat{\mathbf{E}}$ and $\mathbf{X}_n\mathbf{E}$, denoted $r_1,\ldots,r_s$, and the average of these squared correlations.  That is,
  \begin{equation}
      \overline{r^2}=\frac{1}{s}\sum^s_{j=1}r_j^2=\frac{1}{s}\text{trace}\left[\widehat{\mathbf{E}}^\top\widehat\bSigma\mathbf{E}\left(\mathbf{E}^\top\widehat\bSigma\mathbf{E}\right)^{-1}\mathbf{E}^\top\widehat\bSigma\widehat{\mathbf{E}}\left(\widehat{\mathbf{E}}^\top\widehat\bSigma\widehat{\mathbf{E}}\right)^{-1}\right].\label{r2}
  \end{equation}

Let $R(.,.)$ denote the functional for the estimator in \eqref{r2} and $E$ the functional associated with the $\widehat{\mathbf{E}}$ estimator.  Then, adapting \eqref{r2} above, we consider $R[E(F),E(F_\epsilon)]=s^{-1}\text{trace}\left[E(F_\epsilon)^\top\bSigma{E(F)}\left({E(F)}^\top\bSigma{E(F)}\right)^{-1}{E(F)}^\top\bSigma{{E(F_\epsilon)}}\left({{E(F_\epsilon)}}^\top\bSigma{{E(F_\epsilon)}}\right)^{-1}\right]$.

Let $\text{MD}_0=\sqrt{(\bx_0-\bmu)^\top\bm{\Sigma}^{-1}(\bx_0-\bmu)}$ be the Mahalanobis distance of $\bx_0$ from the global mean $\bmu$ scaled according to the global covariance matrix $\bSigma$.

\begin{theorem}\label{thereom:IFrho}
When $y_0=k$ such that $\bx_0$ is a contaminant in the $k$th subpopulation, the IF for $\rho$ is equal to
$$\IF(\rho,F;\bx_0)=\frac{1}{s}\left(\text{MD}_0^2-\sum^s_{j=1}\frac{z_j^2}{1+\lambda_j}\right)\sum^s_{j=1}\frac{\left[z_j - (1+\lambda_j)\overline z_{jk}\right]^2}{\lambda_j^2(1+\lambda_j)}$$
where $\text{MD}_0$ is defined above, and $z_j=\be_j^\top(\bx_0-\bmu)$ and $\overline{z}_{jk}=\be_j^\top(\bx_0-\bmu_k)$.
\end{theorem}

Theorem \ref{thereom:IFrho} offers some interesting insights into the types of observations that can and cannot influence LDA results, as well as why this influence measure has an advantage over using the influence function for the directions themselves.  Consider the following cases.

\subsubsection*{Case 1: Contaminant equal to the mean}  For $\bx_0=\bmu$, $\IF(\rho,F;\bx_0)=0.$  This is an interesting case since $\bmu$ is the overall mean.  Hence an $\bx_0=\bmu$ may be not at all like the average $\bx$ from each of the subpopulations. This does not mean that there is zero influence on the individual directions, however.    From Theorem \ref{theorem:IFs} we have that
$$\IF(e_j,F;\bx_0) = \frac{1}{2}\be_j-\sum^s_{r=1,r\neq j}\frac{(1+\lambda_j)}{(\lambda_j-\lambda_r)}\overline{z}_{rk}\overline{z}_{jk}\be_r$$
which is an element of Span$(\be_1,\ldots,\be_s)$.  Hence, while the influence function for the $j$th direction is not zero, the approximate change due to $\bx_0$ is not harmful since the resulting direction is capturing the information that separates the groups.  This is the reason why, as noted above, $\IF(\rho,F;\bx_0)=0.$

\subsubsection*{Case 2: $\bx_0-\bmu$ orthogonal to the $\be_j$s} From Theorem \ref{theorem:IFs} we have that $\IF(e_j,F;\bx_0)=\be_j/2$ which is bounded.  Influence is limited even for large outliers when $\bx_0-\bmu$ is orthogonal to the discriminant directions.  Additionally, the direction remains the same so that the limited influence is not harmful.  This is also indicated by the fact that $\IF(\rho,F;\bx_0)=0$.

\subsubsection*{Case 3: $\bx_0-\bmu \in \text{Span}(\be_1,\ldots,\be_s)$}  While the $\bx_0$s above have no influence, this is not the case for $\bx_0-\bmu$ when it is in the same direction as an $\be_j$, or when it is an element of the span of the $\be_j$s.  In both cases the IFs are unbounded suggesting that outliers of this type can be harmful to the LDA estimators.

\section{Wine data example}
 
 We now consider an example of LDA applied to the wine data set obtained from
 \begin{center}
     \url{http://archive.ics.uci.edu/ml/machine-learning-databases/wine/wine.data}
 \end{center}
 The data consists of 13 predictor variables that resulted from the chemical analyses of Italian-grown wine from $g=3$ cultivars.
 
\begin{figure}[h!t]
\centering
\includegraphics[scale=0.9]{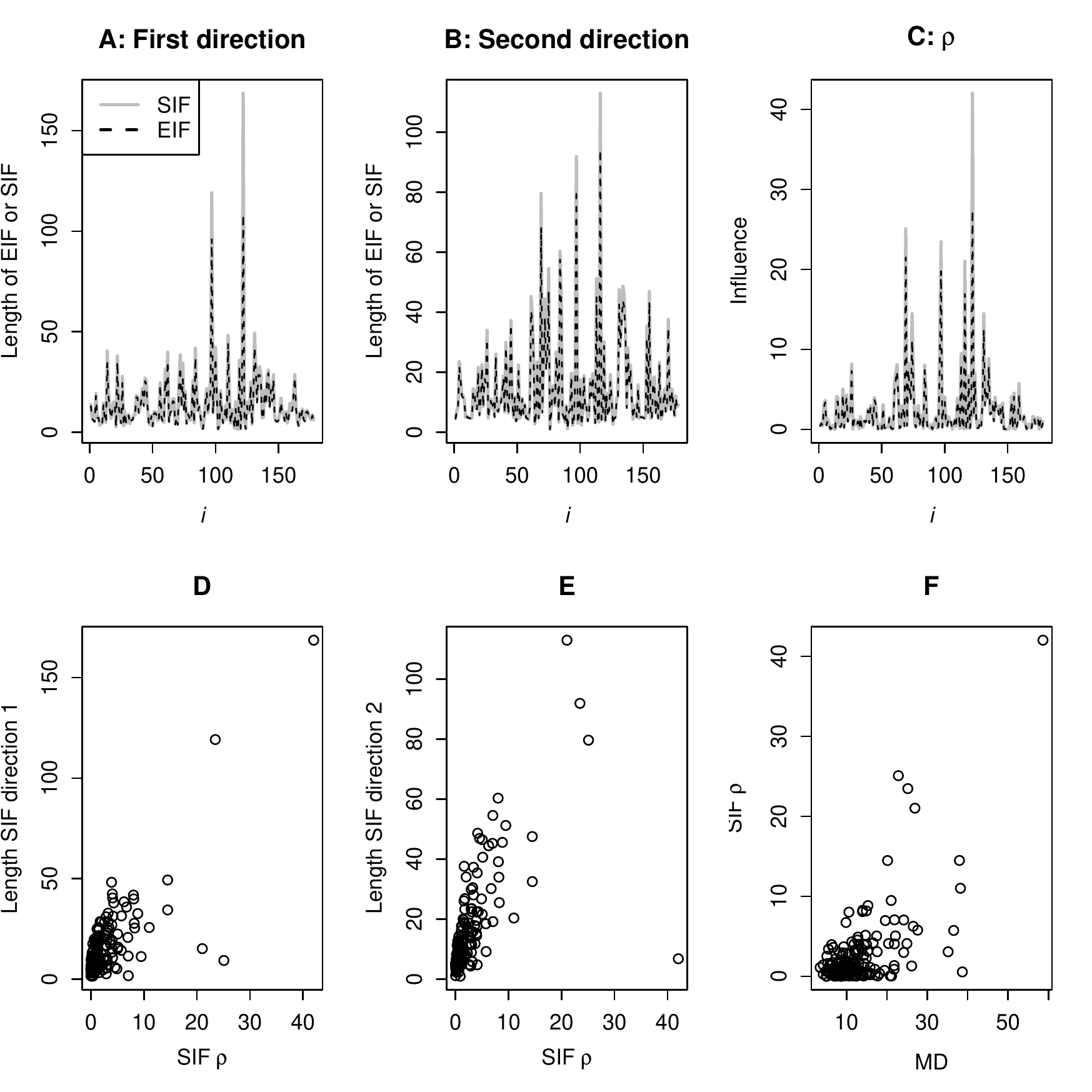}
\caption{Comparisons between the sample influence function (SIF) and empirical influence function (EIF) values - Plots A and B: squared lengths of influence function vectors for the first and second directions; Plot C: influence functions for $\rho$), the SIF for each direction and influence for $\rho$ (Plots D and E) for the wine data set and comparison between the SIF for the directions and the IF for $\rho$ and a comparison of the Mahalanobis distances for the predictor vectors and the SIF (Plot F).}\label{figure:WineIF}
\end{figure}

In Plots A and B of Figure \ref{figure:WineIF} we compare the SIF and EIF from Theorem \ref{theorem:IFs} of the first and second direction estimators for the wine data.  We have used the length of the influence vector for each observation as the comparison.  As can be seen from the plots, the EIF provides a close approximation to the SIF values.  In Plot C we compare the SIF and EIF for the IF in Theorem \ref{thereom:IFrho} for the $\rho$ measure.  Here, the SIF is equal to $(n - 1)(1 - \overline{r^2_i})$ for each $i=1,\ldots, n$ where $\overline{r^2_i}$ is the average squared canonical correlations comparing the two directions based on estimation from the complete data set and the directions with the $i$th observation removed.  Again, the EIF provides a very good approximation to the SIF.  In Plots D and E we compare the length of the direction of the SIF vectors and the $\rho$ SIF.  There is one comparatively highly influential observation that exerts a strong influence on the first direction, but not the second.  Finally, we plot the Mahalanobis distances (MDs) and the SIF for $\rho$.  While the largest outlier, as indicated by the MD, is also the most influential observation, the second largest outlier exerts little influence.  On the other hand, the second most influential observation had  the 15th largest MD.  Hence, a consideration of outliers may or may not be useful in detecting observations harmful in estimation.

\begin{figure}[h!t]
\centering
\includegraphics[scale=0.9]{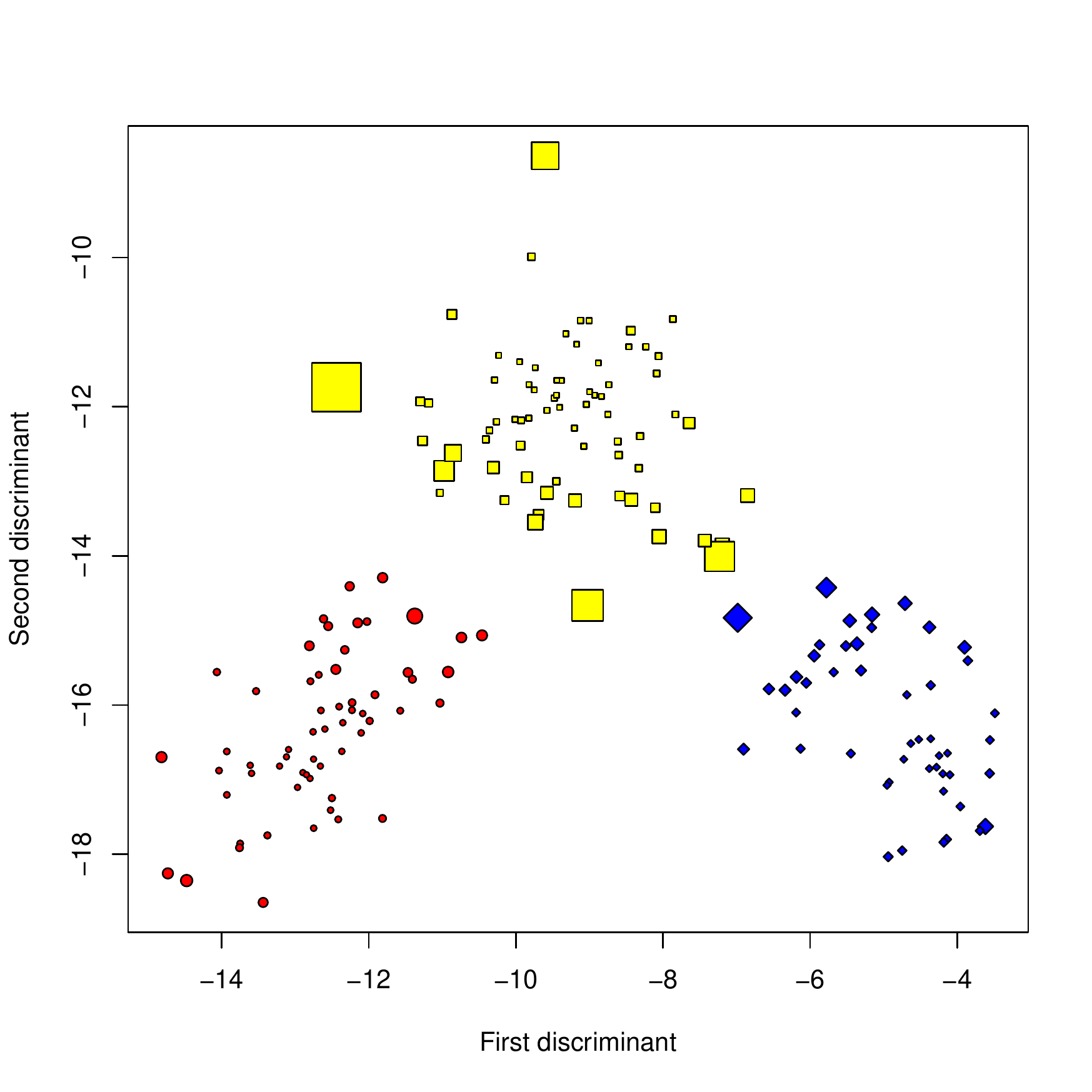}
\caption{A biplot of the linear discriminants where color identifies the three cultivars.  The size of the point is proportional to the size of the SIF for the $\rho$ measure.}\label{figure:WineIF2}
\end{figure}

In Figure \ref{figure:WineIF2} we provide a biplot of the discriminants for the wine data where color identifies the three different cultivars.  The size of the points is proportional to the size of the SIF for the $\rho$ measure.  While influential observations are usually seen on the outskirts of the their respective subpopulation groupings, others, even nearby influential observations, are not necessarily influential.  Again, and as with the case of the outlier considerations, the influence diagnostics could be employed in practice to find potential problematic observations.

\section{Discussion}

In this paper we have obtained influence functions for LDA that arise from the link between LDA and SIR, and then the dimension reduction influence functions for methods such as SIR that have been published in the literature.  We provided influence functions for individual discriminant directions as well as an overall measure that considers all of the directions.  The latter is useful since an influential observation may not change the discriminant directions themselves in which case it is not harmful to the analysis.  The influence functions can be used to explore robustness properties of LDA estimators as well as to create influence diagnostics that can be used in practice.  We applied the sample-based influence functions to a real data set and highlighted that outliers are not necessarily harmful and that, conversely, influential observations need not be outliers.    

\appendix

\section*{Appendix}

The influence functions within this paper have been derived via existing results presented in the literature for SIR.  This appendix will therefore begin with a brief overview of SIR and more details on the link between SIR and LDA.  Technical details for influence functions will follow which include influence functions for SIR in the LDA setting of a categorical response. 

\section{A brief overview of Sliced Inverse Regression}

\cite{LI91} considered the model
\begin{equation}
Y = f(\bm{\beta}_1^\top\mathbf{x},\ldots, \bm{\beta}_K^\top\mathbf{x},\varepsilon)\label{SIR_model}
\end{equation}
where $Y\in \mathbb{R}$ is a random univariate response, $\mathbf{x}\in \mathbb{R}^p$ is a random p-dimensional vector or predictor variables, $\varepsilon\in \mathbb{R}$ is a random error term independent of $\mathbf{x}$ and with $E(\varepsilon)=0$, and $f$ is an unknown link function.  Sample realizations of $y$ and $\mathbf{x}$ are available, denoted $\{y_i,\mathbf{x}\}^n_{i=1}$ which may be used to seek regression information regarding $E(Y|\mathbf{x})$ with the form of the link function $f$ of direct interest.  When $p > 2$, visualization of the $y_i$'s versus the $\mathbf{x}_i$'s is difficult. However if $K < p$ then the $\mathbf{x}_i$'s could be replaced with the $\bm{\beta}_k^\top\mathbf{x}_i$'s without loss of information and resulting in an easier visualization task due to the reduction in dimension of the predictors from $p$ to $K$.

Let $\mB = [\bm{\beta}_1,\ldots,\bm{\beta}_K]$ and consider the following condition:
\begin{condition}\label{condition:LDC} $E(\mathbf{x}|\mB^\top\mathbf{x})$ is linear in
$\mB^\top\mathbf{x}$.
\end{condition} \noindent For $E(\mathbf{x})=\bm{\mu}$, $\Cov(\mathbf{x})=\bm{\Sigma}$ and $\mS$ denoting the space spanned by the columns of $\mB$.  When Condition \ref{condition:LDC} holds, \cite{LI91} showed that for $E(\mathbf{x}|Y)$ denoting the \textit{inverse regression curve}, $\bm{\Sigma}^{-1}\left[E(\mathbf{x}|Y)-\bm{\mu}\right]\in \mS$ so that the inverse regression curve contains information regarding the span of the $\bm{\beta}_k$'s.  Given that $E(\mathbf{x}|Y)$, Li also showed that $\bm{\Sigma}^{-1}\left[\bm{\mu}_S-\bm{\mu}\right]\in \mS$ where $\bm{\mu}_S=E(\mathbf{x}|Y\in S)$ is known as a slice mean and where $S$ is a sub-interval of $Y$.

The SIR matrix is then $\mathbf{V}=\bm{\Sigma}^{-1/2}\sum^H_{h=1}p_h(\bm{\mu}_h-\bm{\mu})(\bm{\mu}_h-\bm{\mu})^\top\bm{\Sigma}^{-1/2}$ where $\bm{\mu}_h=E(\bm{x}|Y\in S_h)$ with $\bigcup^H_{h=1}S_h=\text{range}(Y)$ and $p_h=P(Y\in S_h)$.  Under Condition \ref{condition:LDC}, $\mathbf{V}$ has at most $K$ non-zero eigenvalues whose corresponding eigenvectors are elements of $\bm{\Sigma}^{1/2}\mS$.

In practice it is simple to estimate $\mathbf{V}$.  Commonly, one settles on a number of slices $H$ and partitions the $\mathbf{x}_i$'s into $H$ approximately equal sample size slices based on the order of the $y_i$'s.  The estimated $\bm{\mu}_h$'s are then the sample means of the $\mathbf{x}_i$'s in the $h$th slice.  Eigenvectors corresponding to the $K$ largest eigenvalues of the estimated $\mathbf{V}$ form an estimated basis for $\bm{\Sigma}^{1/2}\mS$ which can be re-scaled to form an estimated basis for $\mS$.  There exist several methods for choosing $H$ and deciding on a suitable $K$.  For example, see \cite{LI&SA12} who simultaneously do both. 

\section{The link between LDA and SIR}\label{sect:LDA_and_SIR}

Starting with \cite{KE91}, several authors have pointed out links between SIR and LDA.  \cite{CO&YI01} provided a formal study of SIR in the setting of the categorical response with some discussion on the links with LDA, and others have applied SIR for classification in areas such as DNA microarray analysis \citep[e.g,][]{BU&PF03,DA&LI&RO06}.  Let $\widehat{\lambda}_j$ denote the $j$th LDA eigenvalue estimate and $\widehat{\alpha}_j$ denote the respective SIR estimate.  \cite{LI00} provided a formal link between the eigenvalue estimates equal to
\begin{equation}
\alphahat_j = \frac{\lambdahat_j}{1+ \lambdahat_j}\;\;\text{and}\;\;\lambdahat_j=\frac{\alphahat_j}{1 - \alphahat_j}\;\;\;(j=1,\ldots,s).\label{eigenvalue_link}
\end{equation}

With $\widehat{\be}_j$ and $\widehat{\bv}_j$ denoting the eigenvector estimates that correspond to the eigenvalues above for LDA and SIR respectively, we can also show that
\begin{equation}
\widehat{\bv}_j = \sqrt{\left(\frac{n}{n-g}\right)\cdot\left(\frac{1}{1+\lambdahat_j}\right)} \cdot \widehat{\be}_j\;\;\text{and}\;\;\widehat{\be}_j=\sqrt{\left(\frac{n-g}{n}\right)\cdot\left(\frac{1}{1-\alphahat_j}\right)}\cdot\widehat{\bv}_j ,\;\;\;(j=1,\ldots,s).
\end{equation}

The terms involving $n$ and $g$ above arise due to the LDA directions being scaled such that $\widehat{\be}_j^\top\bS_p\widehat{\be}_j=1$ where $\bS_p=\bW/(n-g)$ is the sample pooled estimate of the common variance.  However, if one were to use $\bW/n$ as the estimate, then this term is equal to one.  At the population level, that is at $F=F_\infty$, we have that
\begin{equation}
{\bv}_j = \sqrt{\frac{1}{1+\lambda_j}} \cdot {\be}_j\;\;\text{and}\;\;{\be}_j = \sqrt{\frac{1}{1-\alpha_j}} \cdot {\bv}_j,\;\;\;(j=1,\ldots,s)\label{eigenvector_link}
\end{equation}
since $\lim_{n\rightarrow \infty} n/(n-g) = 1$.

\section{Technical derivations}

In this section we will provide several derivations of influence functions for SIR in the categorical response setting before providing the proofs for Theorems 1 and 2.  We derive the influence functions for SIR since it is simpler to derive these first before using the links in Section \ref{sect:LDA_and_SIR} to compute the LDA counterparts.

\subsection{Influence functions for SIR for a discrete response}

Throughout let $G$ be an arbitrary distribution function. Let $C$ be the functional for the covariance matrix estimator where $C(G)=\int (\bx-\bmu)(\bx-\bmu)^\top dG$, $C(F)=\bSigma$ and $C(F_n)=\widehat{\bSigma}$.  Then $\IF(C;F) = (\bx_0-\bmu)(\bx_0-\bmu)^\top-\bSigma$ \citep[see, for e.g.,][]{CR85}.

\cite{PR05} derived the influence function for the SIR matrix estimator in a slightly different context to that required here.  A continuous response variable was assumed where slices were defined by choosing slice proportions (i.e. the $p_j$'s) which in turn leads to boundaries on the domain of the response defining the slices.  Contamination then also affects these boundaries defining the slices.  For a discrete response, as we have here, the boundaries themselves do not change due to contamination, however the proportions within each group do.  Below we present the influence function for this setting which agrees with the result from \cite{PR05} with the exception that a nuisance term (arising from the moving boundaries in the continuous case) is missing.

\begin{lemma}\label{lemma:IFV}
Suppose that $y_0=k$ identifies the contaminant as being in the $k$th subpopulation.  Let $V$ denote the functional for the SIR matrix estimator such that $V(F)=\bSigma^{-1/2}\sum^g_{i=1}p_i(\bmu_i-\bmu)(\bmu_i-\bmu)^\top\bSigma^{-1/2}$.  Then
\begin{align*}
\IF(V;F)=&\IF(C^{-1/2};F)\bSigma^{1/2}V(F) + V(F)\bSigma^{1/2}\IF(C^{-1/2};F)+\bz_j\bz_j^\top - \overline{\bz}_{jk}\overline{\bz}_{jk}^\top - V(F)
\end{align*}
where $\IF(C^{-1/2};F)$ is the influence function for the inverse square root of the covariance matrix estimator whose functional is $C$, $\overline{\bz}_{jk}=\bSigma^{-1/2}(\bx_0-\bmu_{k})$ and $\bz_j=\bSigma^{-1/2}(\bx_0-\bmu)$.
\end{lemma}
\begin{proof}
Let $\mT$ be the functional for the sample mean estimator where $\mT(G)=\int \bx dG=\bmu$.  Similarly, let $\mT_i$ be the functional for the sample mean estimator associated with the $i$th group.  Firstly, we have that $\mT(F_\epsilon)=\int \bx d[(1-\epsilon)F+\epsilon\Delta_{\bx_0}]=(1-\epsilon)\bmu + \epsilon\bx_0$ so that the influence function for the sample mean estimator is $\IF(\mT;F)=\bx_0-\bmu$.  For the mean of the $i$th group, the estimator is only influenced if $y_0=i$ so that
\begin{equation}
\mT_i(F_\epsilon)=\bmu_i+\frac{\epsilon}{p_i(\epsilon)}I(y_0=i)(\bx_0-\bmu_i),\nonumber
\end{equation}
where $I(y_0=i)$ equals one if $y_0=i$ and zero otherwise, and where $p_i(\epsilon)=(1-\epsilon)p_i+I(y_0=i)\epsilon$ is the proportion in the $i$th group following contamination.  Consequently we have
$$V(F_\epsilon)=C^{-1/2}(F_\epsilon)\sum^g_{i=1}p_i(\epsilon)\left[\mT_i(F_\epsilon)-\mT(F_\epsilon)\right]\left[\mT_i(F_\epsilon)-\mT(F_\epsilon)\right]^\top C^{-1/2}(F_\epsilon)$$
and the resulting influence function follows simply by deriving $\left[\partial V(F_\epsilon)/(\partial\epsilon)\right]\big|_{\epsilon = 0}$ and rearranging.
\end{proof}

In order to derive the influence function for the LDA directions, we also need the influence function for the SIR eigenvalue estimators.  Using previous results, it is relatively simple to compute these (compared to the LDA eigenvalue estimators) due to the SIR matrix being symmetric.  We present the influence function for the $j$th SIR eigenvalue estimator in the below lemma.

\begin{lemma}\label{lemma:IFalpha}
Again suppose that $y_0=k$.  Let $a_j$ denote the functional for the $j$th SIR eigenvalue estimator such that $a_j(F)=\alpha_j$ and where $j \leq s$.  Then
$$\IF(a_j;F)=(1 -\alpha_j)w_j^2 - \overline{w}_{jk}^2$$
where $w_j = \mathbf{v}_j^\top(\bx_0-\bmu)$ and $\overline{w}_{jk}=\bv_j^\top(\bx_0-\bmu_{k})$.
\end{lemma}
\begin{proof}
  Let $\eta_j$ be the functional for the eigenvector of the SIR matrix corresponding to the $j$th ordered eigenvalue (which is $a_j(F)=\alpha_j$ at $F$).  Then, using the Chain Rule and the identity $[\eta_j(G)]^\top V(G) \eta_j(G)=a_j(G)$, we have
$$\IF(a_j;F) = 2\alpha_j [\eta_j(F)]^\top \IF(\eta_j;F) + [\eta_j(F)]^\top \IF(V;F)  \eta_j(F).$$  From Equation (13) of \cite{PR05}, $[\eta_j(F)]^\top \IF(\eta_j;F)=0$ so that the influence function reduces to $\IF(a_j;F) = [\eta_j(F)]^\top \IF(V;F) \eta_j(F)$ where $\IF(V;F)$ is provided in Lemma \ref{lemma:IFV}.  We therefore have
\begin{align}
\IF(a_j;F) =& 2\alpha_j[\eta_j(F)]^\top\IF(C^{-1/2};F)\bSigma^{1/2}\eta_j(F)+w_j^2 -\overline{w}_{jk}^2-\alpha_j.
\end{align}

Using derivations following Equation (19) of \cite{PR05}, it can be shown that $2\alpha_j[\eta_j(F)]^\top\IF(C^{-1/2};F)\bSigma^{1/2}\eta_j(F)=\alpha_j[\eta_j(F)]^\top\bSigma^{1/2}\IF(C^{-1};F)\bSigma^{1/2}\eta_j(F)$ where $\IF(C^{-1};F)=-\bSigma^{-1}\IF(C;F)\bSigma^{-1}$ for the $\IF(C;F)$ provided at the start of this section.  Application of this result, and recalling $\bSigma^{-1/2}\eta_j(F)=\bv_j$ and $\|\eta_j(F)\|=1$, completes the proof.
\end{proof}

Recall from Lemma \ref{lemma:IFV} that this influence function for SIR with a discrete response is the same as that for the continuous case in \cite{PR05} except for a nuisance term that is missing.  Consequently, we use the influence function results for the SIR directions for a continuous response and conveniently drop the nuisance term.  Let $v_j$ denote the functional for the $j$th SIR direction estimator.  Then, from \cite{PR05} and using our notations and with some rearranging of terms, we have
\begin{align}
\IF(v_j; F)=& \frac{1}{2}\left[w_j^2+1-2\frac{(w_j-\overline{w}_{jk})w_j}{\alpha_j}\right]\bv_j+\frac{1}{\alpha_j}\left[(1-\alpha_j)w_j-\overline{w}_{jk}\right]\bSigma^{-1}(\bx_0-\bmu)\nonumber\\
& + \sum^s_{r=1,r\neq j}\frac{1}{\alpha_j-\alpha_r}\left[\frac{\alpha_r(1-\alpha_j)}{\alpha_j}w_rw_j +  \frac{(\alpha_j - \alpha_r)}{\alpha_j}w_r\overline{w}_{jk}-\overline{w}_{rk}\overline{w}_{jk}\right]\bv_r.\label{IFvj}
\end{align}

Finally, \cite{PR&SM10} gave an influence function for a class of dimension reduction models which includes SIR.  The reasoning for this particular influence function is provided in Section \ref{section:IFs_for_span}.  Here we derive the form of the influence function for SIR which in turn is then used to obtain the influence function associated with LDA.

\begin{lemma}\label{lemma:IFrho}
Using previous notations, the influence function based on the average squared canonical correlations for SIR is equal to
$$\IF(\rho;F)=\frac{1}{s}\left(\text{MD}_0^2-\sum^s_{j=1}w_j^2\right)\cdot\sum^s_{j=1}\frac{1}{\alpha_j^2}\left[(1-\alpha_j)w_j-\overline{w}_{jk}\right]^2$$
where $\text{MD}_0=\sqrt{(\bx_0-\bmu)^\top\bm{\Sigma}^{-1}(\bx_0-\bmu)}$.
\end{lemma}
\begin{proof}
From Theorem 2 of \cite{PR&SM10},
\begin{equation}
\IF(\rho;F) = \frac{1}{s}\sum^s_{j=1}\left\|(\mathbf{I}-\mathbf{P})\bm{\Sigma}^{1/2}\IF(v_j; F)\right\|^2
\end{equation}
where $\mathbf{P}=\sum^s_{j=1}\bm{\Sigma}^{1/2}\bv_j\bv_j^\top\bm{\Sigma}^{1/2}$ which is a projection matrix onto the space spanned by $\bm{\Sigma}^{1/2}\bv_1,\ldots,\bm{\Sigma}^{1/2}\bv_s$.  Consequently, $\mathbf{I}-\mathbf{P}$ is a projection matrix onto the compliment of this space so that $(\mathbf{I}-\mathbf{P})\bSigma^{1/2}\bv_j=0$ for all $j=1,\ldots,s$.  Therefore, from \eqref{IFvj}, we may write
$$\IF(\rho;F)=\frac{1}{s}\left\|(\mathbf{I}-\mathbf{P})\bSigma^{-1/2}(\bx_0-\bmu)\right\|^2\cdot\sum^s_{j=1}\frac{1}{\alpha_j^2}\left[(1-\alpha_j)w_j-\overline{w}_{jk}\right]^2.$$

The proof is complete by noting that, since $\mathbf{I}-\mathbf{P}$ is idempotent, it follows that
$$\left\|(\mathbf{I}-\mathbf{P})\bSigma^{-1/2}(\bx_0-\bmu)\right\|^2=(\bx_0-\bmu)^\top\bSigma^{-1}(\bx_0-\bmu)-\sum^s_{j=1}(\bx_0-\bmu)^\top\bv_j\bv_j^\top(\bx_0-\bmu)$$
where $(\bx_0-\bmu)^\top\bv_j=\bv_j^\top(\bx_0-\bmu)=w_j$.
\end{proof}

\subsection{Proof of Theorem 1}

From \eqref{eigenvalue_link}, we may write the $j$th LDA eigenvalue functional in terms of the SIR eigenvalue functional as $l_j(F)=a_j(F)/[1-a_j(F)]$.  Therefore, using the Chain Rule,
$$\frac{\partial}{\partial\epsilon}l_j(F_\epsilon)=\frac{a_j(F_\epsilon)}{[1-a_j(F_\epsilon)]^2}\frac{\partial}{\partial\epsilon}a_j(F_\epsilon) + \frac{1}{1-a_j(F_\epsilon)}\frac{\partial}{\partial\epsilon}a_j(F_\epsilon). $$  Hence the $\IF(l_j,F;\bx_0)=[\partial l_j(F_\epsilon)/(\partial\epsilon)]|_{\epsilon=0}=\IF(a_j,F;\bx_0)/(1-\alpha_j)^2$.  The result for the LDA eigenvalue estimator follows from Lemma \ref{lemma:IFalpha} and from noting the relationship between the LDA and SIR eigenvalues and eigenvectors from \eqref{eigenvalue_link} and \eqref{eigenvector_link}.

\subsection{Proof of Theorem 2}

Theorem 2 can be obtained from Lemma \ref{lemma:IFrho} and using $w_j=z_j/\sqrt{1+\lambda_j}$, $\overline w_{jk}=\overline z_{jk}/\sqrt{1+\lambda_j}$ and $\alpha_j=\lambda_j/(1+\lambda_j)$.

\begin{footnotesize}
\bibliography{ref}
\bibliographystyle{authordate4}
\end{footnotesize}

\end{document}